\documentclass[12pt,reqno]{amsart}
\usepackage{graphicx,amssymb,amscd,mathrsfs}
\newtheorem{theorem}{Theorem}
\newtheorem{theoreme}{Theorem}

\newtheorem{theoremp}{Theorem}
\newtheorem{theoremr}{Theorem}
\newtheorem{theoremno}{Theorem}
\newtheorem{examp}[theoreme]{Example}

\newtheorem{prop}[theoremp]{Proposition}
\newtheorem{rk}[theoremr]{Remark}
\newtheorem{dfn}[theoremno]{Definition\!\!}

\renewcommand{\thetheoremno}

\newcommand\bib[1]{\bibitem[#1]{#1}}
\newcommand{\comm}[1]{}
\newcommand\1{{\bf 1}}
\renewcommand\a{\alpha}

\renewcommand\d{\delta}
\newcommand\D{{\mathcal D}}

\newcommand\E{{\mathcal E}}

\newcommand\g{\mathfrak{g}}

\newcommand\GG{{\mathcal G}}

\newcommand\Ll{{\let\mathcal\mathscr\mathcal L}}

\newcommand\m{\mathfrak{m}}

\newcommand\op[1]{\mathop{\rm #1}\nolimits}
\newcommand\ot{\otimes}

\newcommand\R{{\mathbb R}}
\newcommand\s{\mathfrak{s}}
\renewcommand\S{{\mathcal S}}

\begin{document}

 \title[Symmetries and filtered Lie equations]{Symmetries of filtered structures \\
        via filtered Lie equations}
 \author{Boris Kruglikov}
 \maketitle

  \vspace{-15pt}
 \begin{abstract}
We bound the symmetry algebra of a vector distribution, possibly equipped with an additional
structure, by the corresponding Tanaka algebra. The main tool is the theory of weighted jets.
 \end{abstract}

 \section{Introduction and the main result}\label{S1}

Consider a manifold $M$ and a non-holonomic (bracket generating) vector distribution $\Delta\subset TM$,
possibly equipped with an additional structure, like sub-Riemannian metric or conformal or CR-structure.
A specification of these filtered structures is provided below.

Given such a structure, a sheaf of graded Lie algebras $\g=\oplus\g_i$ is naturally associated with it.
If we consider only the distribution $\Delta$, then $\m(x)=\g_-(x)$ is the well-known graded nilpotent
Lie algebra (GNLA: nilpotent approximation or Carnot algebra) at $x\in M$, and $\g(x)$
is its Tanaka prolongation (Tanaka algebra). If an additional structure on $\Delta$
is given, then $\g_0$ or some higher $\g_i$ ($i>0$) is reduced and the algebra is further prolonged.
In any case for a filtered structure $\mathcal{F}$ on $M$ we associate its sheaf of Tanaka algebras
$\g(x)$, $x\in M$.

 \begin{theorem}\label{thm1}
The symmetry algebra $\S$ (possibly infinite-dimensional) of a filtered structure $\mathcal{F}$
has the natural filtration with the associated grading $\s$ naturally injected
into $\g(x)$ for any regular point $x\in M$. In particular,
 $$
\dim\S\le\sup_M\dim\g(x).
 $$
Provided that the filtered structure is of finite type ($\g_\kappa(x)=0$ for some
$\kappa>0$ and all $x\in M$),
the right-hand-side can be changed to $\inf\limits_M\dim\g(x)$.
 \end{theorem}

If $\Delta$ is considered without an extra structure, then this statement was
proved in \cite{K1} by studying the Lie equation considered as a submanifold in the usual jet-space.
In addition to the above mentioned type structures we can impose curvature of the structure
as a reduction of $\g$ (see Remark \ref{rk3} in Section \ref{S3} on the issue of regularity),
thus essentially restricting the symmetry algebra, yielding a bound
on the gap between the maximal and the next (submaximal) dimensions of the possible symmetry
algebras of the given type structures \cite{K2}.
In the context of parabolic geometries (in the complex-holomorphic or split-real smooth cases) 
this gap was fully computed using the above idea in \cite{KT}.

Regularity for the points $x$ in Theorem \ref{thm1} is defined via the Lie equations in Section \ref{S4}.
When $\sup\dim\g(x)$ is finite (that is $\mathcal{F}$ is of finite type)
or the filtered structure $\mathcal{F}$ is analytic, then the set of regular
points is open and dense; in general a generic point is regular.

In this paper we consider weighted jets and relate them to Tanaka algebras.
On this way we obtain another proof of Theorem 1 of \cite{K1} and get a more general result.

 \section{Tanaka algebra of a distribution with a structure}\label{S2}

Given a distribution $\Delta$ its \emph{weak derived flag} $\{\Delta_i\}_{i>0}$, with $\Delta_1=\Delta$,
is given via the module of its sections by $\Gamma(\Delta_{i+1})=[\Gamma(\Delta),\Gamma(\Delta_i)]$.
The distribution $\Delta$ will be assumed \emph{completely non-holonomic}, meaning there exists a natural
number $\nu$ such that $\Delta_\nu=TM$.

The quotient $\g_i=\Delta_{-i}/\Delta_{-i-1}$ (we let $\Delta_0=0$) evaluated at the points $x\in M$ is not a vector bundle in general (rank needs no be constant); however its local sections form a sheaf and
the module of its global sections will be denoted by $\Gamma(\g_i)$ (similarly for other sheafs).

At every point $x\in M$ the vector space $\m=\oplus_{i<0}\g_i$ has a natural structure of
graded nilpotent Lie algebra. The bracket on $\m$ is induced by the commutator
of vector fields on $M$. $\Delta$ is called \emph{strongly regular} if the GNLA $\m=\m_x$ does not depend
on the point $x\in M$.

The Tanaka prolongation $\g=\hat\m$ is the graded Lie algebra with negative graded part $\m$
and non-negative part defined successively by (see [AK] for discussion and interpretation of this prolongation)
 $$
\g_k=\{u\in\bigoplus\limits_{i<0}\g_{k+i}\ot\g_i^*:
u([X,Y])=[u(X),Y]+[X,u(Y)],\ X,Y\in\m\}.
 $$
Since $\Delta$ is bracket-generating, the algebra $\m$ is fundamental, i.e.
$\g_{-1}$ generates the whole GNLA $\m$, and therefore the grade $k$ homomorphism $u$
is uniquely determined by the restriction $u:\g_{-1}\to\g_{k-1}$.

At every point $\g=\oplus\g_i$ is naturally a graded Lie algebra, called
the {\em Tanaka algebra\/} of $\Delta$ (the bracket is induced by the commutator of vector fields).
To indicate dependence on the point $x\in M$, we will write $\g=\g(x)$ or $\g_x$
(also the value of a vector field $Y$ at $x$ will be denoted by $Y_x$).

In addition to introducing the Lie algebra $\g$, which majorizes the symmetry algebra of $\Delta$,
the paper \cite{Ta} contains the construction of an important ingredient to the equivalence problem
-- an absolute parallelism on the prolongation manifold $\GG_i$ of the structure, provided it is
strongly regular.

Distribution is locally flat if the structure functions of the absolute parallelism vanish.
Then the distribution $\Delta$ is locally diffeomorphic
to the standard model on the Lie group corresponding to $\m$, see \cite{Ta}.

The prolongation manifolds are the total spaces of the bundles $\GG_i\to M$ (in the strongly regular case).
For instance, the fiber of $\GG_0$ over $x$ consists of all grading preserving isomorphisms
of Lie algebras $u_0:\m\to\m_x$, where $\m$ is an abstract GNLA of the same type as $\m_x$.
Denoting by $\op{Aut}_0(\m)$ the group of grading preserving automorphisms of $\m$,
we conclude that $\GG_0$ is a principal $\op{Aut}_0(\m)$-bundle over $M$;
the tangent to the fiber is the Lie algebra $\mathfrak{der}_0(\m)$ of grading preserving
derivations of $\m$. The fiber of $\GG_1$ over $u_0\in\GG_0$
consists of the adapted frames $u_1:\g_{-1}\to\g_0$ that uniquely extend to the grading $-1$ maps
$u_1:\m\to\m\oplus\g_0$ etc, see \cite{Z} for details.

The idea of constructing the frame bundle can be pushed to the general non-strongly regular case.
Here an abstract reference algebra $\m$ is lacking\footnote{This manifests lacking of a canonical
absolute parallelism, though it is possible to construct a parallelism respected by any automorphism.}
and the 0-frames are grading preserving isomorphisms $u_0:\m_x\to\m_x$.
Thus $\GG_0$ is a fiber bundle (having a distinguished 'identity' section), with
the fiber over $x$ being the graded group $\op{Aut}_0(\m_x)$. Again the fiber of $\GG_1$
over $u_0(x)\in\GG_0$ is parametrized by the adapted frames $u_1:\g_{-1}(x)\to\g_0(x)$ etc.
The fibers of $\GG_i\to\GG_{i-1}$ for $i>0$ are isomorphic to $\g_i$.

Notice that in general, the prolongation manifolds $\GG_i$ can be singular, 
when the fibers have varying dimensions.
The structure is called {\em regular\/} at the point $x$ if for every $i$ the spaces $\g_i(y)$ have constant ranks and vary smoothly in $y$ from a neighborhood of $x$ (the size of which can depend on $i$).
The structure is called {\em regular\/} if it is regular at every point on $M$.
It is clear that strongly regular distributions have regular prolongations.
For regular distributions the above $\GG_i$ are smooth fiber bundles over $M$.

The following statement is similar to \cite[Lemma~6]{K1} and \cite[Lemma 4.2.4]{KT},
so its proof is omitted.
  \begin{prop}\label{prop1}
For every $i$ the set of points $x$, where $\op{rank}\g_i(x)$ is locally constant is open and
dense in $M$. \qed
  \end{prop}
Thus if the structure is of finite type,
i.e. $\g_\kappa(x)=0$ for some $\kappa>0$ and all $x\in M$, then the structure is regular on an
open dense set. Otherwise it is regular only on a set of the second category, but in fact assuming
analyticity one can show regularity on an open dense set.

 \section{Filtered structures as reductions of the derived flag}\label{S3}

The best known reduction of the filtered structure related to a strongly regular distribution
$\Delta$ is defined as follows. Let $\tilde{G}_0\subset\op{Aut}_0(\m)$ be a subgroup with
the Lie algebra $\tilde{\g}_0\subset\g_0=\mathfrak{der}_0(\m)$.
A structure of the type $(\m,\tilde{\g}_0)$ (more appropriate would be to write $(\m,\tilde{G}_0)$) is
a principal sub-bundle of $\GG_0\to M$ with the structure group $\tilde{G}_0$.

We can relax the strong regularity assumption by requiring that $\g_i(x)\subset\m_x$ for $i<0$
and $\g_0(x)$ have $x$-independent ranks and then that the bundle $\GG_0\to M$ and its
reduction are smooth (the choice of $\tilde{G}_0$ should also depend smoothly on $x$).
This is a milder regularity assumption in this case.

 \begin{examp}\label{ex1}\rm
Sub-Riemannian structure is a field of Riemannian structures $g$ on the distribution
$\Delta=\g_{-1}$. It is equivalent to a reduction of the structure group to a (subgroup of)
the orthogonal group of $\g_{-1}$, implying the reduction to $\tilde{\g}_0\subset\mathfrak{der}_0(\m)\cap\mathfrak{so}(\g_{-1},g)$.
 A sub-conformal structure is a reduction
of the structure group to a (subgroup of) the linear conformal group $\mathfrak{co}(\g_{-1},[g])$
intersected with $\mathfrak{der}_0(\m)$.
 \end{examp}

 \begin{examp}\label{ex2}\rm
Cauchy-Riemann (CR-) structure is a field of complex structures $J$ on the distribution
$\Delta=\g_{-1}$. It is equivalent to a reduction of the structure group to a (subgroup of)
the complex linear group of $\g_{-1}$ (which shall be thus of even rank), implying the reduction $\tilde{\g}_0\subset\mathfrak{der}_0(\m)\cap\mathfrak{gl}(\g_{-1},J)$.
 \end{examp}

 \begin{rk}
As follows already from these examples, the regularity assumption is also restrictive.
The derived flag of a general sub-Riemannian structure can have varying length
at different points. For a CR structure, for instance of the hypersurface type,
the Levi-flat and generic points can co-exist. However as our main concern is the symmetry algebra,
and it is determined at any regular point, we disregard the most general structures.
 \end{rk}

Similarly, if the bundle $\GG_i$ is smooth for some $i>0$ we can define the reduction
via a $\g_0$-submodule $\tilde{\g}_i\subset\g_i$ (everything smoothly depends on the
point $x\in M$), and prolong the algebra $\g=\g_x$ afterwards, i.e. compute the Tanaka prolongations
 $$
\tilde{\g}_{i+s}=\op{pr}_s(\m,\g_0,\dots,\tilde{\g}_i)=
\{v\in\g_{i+s}:\op{ad}_{\g_{-1}}^s(v)\in\tilde{\g}_i\}
 $$
and define the new Tanaka algebra as
$\g=\g_{-\nu}\oplus\dots\oplus\g_{i-1}\oplus\tilde{\g}_i\oplus\dots$

 \begin{dfn}\label{filtstr}
The filtered structure $\mathcal{F}$ on a manifold $M$ is given by a non-holonomic vector distribution
$\Delta$ and a finite\footnote{Under the assumption of regularity or algebraicity, the finiteness
follows from (a graded version of) Hilbert's basis theorem.}
number of successive reductions of the above bundles $\GG_{i_k}$
(the new reduction concerns the bundles computed from the previous reductions) of
increasing orders $0\le i_1<\dots<i_s$.
Such structure is regular if the corresponding bundles $\GG_i$ are smooth
($\op{rank}\g_i=\op{const}_i$).
 \end{dfn}

Parabolic structures \cite{CS} give examples of filtered structures obtained by reduction of $\g_0$.
For instance, a conformal structure on a manifold $M$ is a reduction of the bundle $\op{End}(TM)$
in the case  $\m=\g_{-1}$ and the prolongation $\g=\g_{-1}\oplus\g_0\oplus\g_1$
is isomorphic to $\mathfrak{so}(p+1,q+1)$ for $n=p+q=\dim\m$.

In fact, by the Yamaguchi prolongation theorem \cite{Y} all but two parabolic geometries
of type $G/P$ for a (complex) simple Lie group $G$ with Lie algebra $\g=\op{Lie}(G)$ 
and a parabolic group $P$ are obtained via a reduction of $\GG_0$.
The two exceptional structures of types $A_n/P_1$ and $C_n/P_1$ 
can be obtained by a higher reduction and prolongation, as in the definition above.

 \begin{examp}\label{ex3}\rm
To obtain projective geometry, which is the parabolic geometry of type $A_n/P_1$
(in the real context $A_n=SL(n+1,\R)$, $P_1$ is the stabilizer of a line),
consider the algebra $\mathfrak{D}_\infty(\R^n)$ of formal vector fields on $V=\R^n$.
It has gradation $\g_{-1}=V$, $\g_0=V^*\ot V=\R\oplus\mathfrak{sl}(V)$, $\g_1=S^2V^*\ot V$, \dots

As a $\g_0$-module $\g_1$ decomposes into irreducible components $\g_1=\g_1'\oplus\g_1''$, where $\g_1'=(S^2V^*\ot V)_0=\op{Ker}(q:S^2V^*\ot V\to V^*)$
with $q$ being the contraction, and $g_1''=V^*\stackrel{i}\to S^2V^*\ot V$ with
$i(p)(v,w)=p(v)w+p(w)v$.

The prolongation of the first reduction is $\g'=\g_{-1}\oplus\g_0\oplus\g_1'\oplus\g_2'\oplus\dots$,
where $\g_k'=\op{Ker}(q:S^{k+1}V^*\ot V\to S^kV^*)$. This is the gradation of the algebra
$\mathfrak{SD}_\infty(\R^n)=\{\xi\in\mathfrak{D}_\infty(\R^n):\op{div}(\xi)=\op{const}\}$.
Indeed, one readily verifies $[\g_{-1},\g_1']=\mathfrak{sl}(V)\subset\g_0$ and $[\g_{-1},\g_{k+1}']=\g_k'$.

This gives one possible reduction of the bundle $\GG_1$. For the other reduction, the prolongation
is trivial and we get $\g''=\g_{-1}\oplus\g_0\oplus\g_1''=\mathfrak{sl}(n+1)$,
which is the grading associated to $A_n/P_1$.
 \end{examp}

 \begin{examp}\label{ex4}\rm
The Lagrangian contact structure gives the parabolic geometry of type $A_n/P_{1,n}$ 
($P_{1,n}$ stabilizes a hyperplane and a line in it) that is not exceptional: 
$\m=\g_{-2}\oplus\g_{-1}$ is the Heisenberg algebra, $\g_0$ is the subalgebra of the 
conformal symplectic algebra $\mathfrak{csp}(2n,\mathbb{R})$ that preserves a 
splitting $\g_{-1}=L_1\oplus L_2$ into a pair of Lagrangian subspaces. 
The Tanaka prolongation $\op{pr}(\m,\g_0)$ coincides with the Lie
algebra $\mathfrak{sl}(n+1)$ of $A_n$. 
If one torsion of this geometry vanishes (i.e. one field of Lagrangian subspaces is integrable), 
then the geometry can be interpreted as the parabolic Monge-Amp\`ere equations  
modulo point transformations.

Let us explain how this structure arises as a reduction 
of a simple infinite Lie algebra from Cartan's list.

Consider the algebra of formal contact vector fields in the standard contact space $\R^{2n-1}=J^1(\R^{n-1})$.
Every such a field is given by a Hamiltonian, so
$\mathfrak{cont}_\infty(\R^{2n-1})\simeq J_0^\infty(\R^{2n-1})$ -- the space of 
formal power series at $0$.
Denoting by $W$ the model symplectic space of $\dim W=2(n-1)$
(contact plane), the gradation of the contact algebra is:
$\g_{-2}=\R$, $\g_{-1}=W$, $\g_0=\mathfrak{csp}(W)=\R\oplus S^2W$, $\g_1=W\oplus S^3W$,
$\g_2=S^2W\oplus S^4W$, \dots

Now let $V\subset W$ be a Lagrangian plane, then $W=V\oplus V^*$ and $S^2W=S^2V\oplus(V\otimes V^*)
\oplus S^2V^*$. Notice that $S^2V,S^2V^*\subset\g_0$ are Abelian subalgebras.

Consider the following subalgebra $\g_0'=\R\oplus(V\otimes V^*)\subset\g_0$.
Its prolongation is $\g'=\g_{-2}\oplus\g_{-1}\oplus\g_0'\oplus\g_1'\oplus\g_2'$, where
$\g_1'=W^*\simeq W$, $\g_2'=\R$. Thus $\g'=\mathfrak{sl}(n+1)$ and the above grading corresponds to $A_n/P_{1,n}$,
the filtered structure being given by a reduction of $\GG_0$.
 \end{examp}

 \begin{rk}
Consider the canonical (Darboux) coordinates $x_i,p_i,u$ on $\R^{2n-1}$, corresponding
to the decomposition $V\oplus V^*\oplus\R$. The polynomial algebra on $\R^{2n-1}$ is weighted by
the rule $\op{w}(x^au^bp^c)=|a|+2(b-1)+|c|$, where $x^a=x_1^{a_1}\cdots x_{n-1}^{a_{n-1}}$,
$|a|=a_1+\dots+a_{n-1}$ for a multi-index $a=(a_1,\dots,a_{n-1})$ and similar for $p^c$ and $|c|$.
This weight is respected by the Jacobi bracket
 $$
\{f,g\}=f\cdot g_u+(f_{x^i}+p_if_u)\cdot g_{p_i}-g\cdot f_u-(g_{x^i}+p_ig_u)\cdot f_{p_i}
 $$
and it gives gradation of the algebra
$\g=\mathfrak{cont}_\infty(\R^{2n-1})\simeq\oplus_{k=0}^\infty S^k(\R^{2n-1})^*$:
$\g_{-2}=\langle 1\rangle$, $\g_{-1}=\langle x_i,p_i\rangle$,
$\g_0=\langle u\rangle\oplus\langle x_ix_j,x_ip_j,p_ip_j\rangle$ (entries in the last component
correspond respectively to $S^2V,V\otimes V^*,S^2V^*$), $\g_1=\langle ux_i,up_i\rangle\oplus
\langle x_ix_jx_k,x_ix_jp_k,x_ip_jp_k,p_ip_jp_k\rangle$, and so forth.
 \end{rk}

 \begin{examp}\label{ex5}\rm
The contact projective structure gives another exceptional parabolic geometry, of type $C_n/P_1$
($C_n=Sp(2n,\R)$).
It is obtained by a reduction of $\GG_1$ from the tower of bundles corresponding to the
formal symmetry algebra $\g$ of the standard contact structure, graded as above.
As a $\g_0$-module $\g_1$ decomposes into irreducible components $\g_1=W\oplus S^3W$.

For the reduction, corresponding to $\g_1'=S^3W$, we get the prolongations $\g_2'=S^4W$, $\g_3'=S^5W$ etc,
so $\g'=\g_{-2}\oplus\g_{-1}\oplus\g_0\oplus\g_1'\oplus\g_2'\oplus\dots$ is the gradation of
the algebra $\mathfrak{Csympl}_\infty(\R^{2(n-1)})$ of formal conformally symplectic vector fields
for $n>2$ (in this case any conformal symplectic transformation is a conformal homothety),
for $n=2$ it is a 1D extension of the algebra $\mathfrak{SD}_\infty(\R^2)$ from Example \ref{ex3}.

For the other choice $\g_1''=W\simeq W^*$ we compute the Tanaka prolongation
$\g_2''=\R$, $\g_3''=0$, whence $\g''=\g_{-2}\oplus\g_{-1}\oplus\g_0\oplus\g_1''\oplus\g_2''
=\mathfrak{sp}(2n,\R)$
and this gradation corresponds to $C_n/P_1$.
 \end{examp}

 \begin{rk}\label{rk3}
If the filtered geometry is strongly regular and $\GG_i$ form the tower of canonical bundles, then
in the limit (on $\GG_\kappa$ in the finite type case or on $\GG_\infty$) we get
the canonical frame, the structure functions of which produce the invariant - curvature $K$
of the structure. The constraint that $K$ (or its projective type) is preserved is another
reduction of the filtered structure, mentioned in the Introduction.
(For conformal Lorenzian 4D structures this reduction corresponds to prescribing the Petrov type of the structure.)

Notice that $K$ needs not have the same type at different points, so this (generalized) filtered
structure is neither an infinitesimal homogeneous geometry in the sense of \cite{GS} nor
a reduction type considered in the above Definition. It is rather similar to
the geometric structures obtained by imposing a higher degree tensor on $\Delta$ or 
to the Finsler structures. The theory of Lie equations developed below is applicable to such structures.
 \end{rk}

 \section{Weighted jets and filtered Lie equations}\label{S4}

Let us start with the structure $\mathcal{F}$ defined by the filtration $\Delta_i$
of the tangent bundle $TM$ given by the derived flag of the distribution $\Delta$.

The spaces $F_{-i}=\Gamma(\Delta_i)$ form the decreasing filtration of the Lie algebra of vector fields
$\D(M)=F_{-\nu}\supset\dots\supset F_{-1}\supset0$
and induce the decreasing filtration of the associative algebra
$D=\op{Diff}(M)$ of scalar differential operators on $M$
($D$ is a left-right bi-module over $C^\infty(M)$):
 \begin{multline*}
D_0=C^\infty(M)\subset D_{-1}=\{v+f:v\in F_{-1},f\in C^\infty(M)\}\subset\\
\subset D_{-2}\subset\dots\subset
D_j=\sum_{i_1+\dots+i_s\ge j}\prod_{t=1}^sF_{i_t}\subset D_{j-1}\subset\dots
 \end{multline*}
This filtration differs from the standard filtration by order.
The algebra $D$ of differential operators with the standard filtration has the
associated graded algebra of symbols evaluated to the bundle $STM=\oplus S^iTM$
(the algebra of polynomials on $T^*M$).
On the contrary the weighted filtration $D_i$ produces the bundle of symbols
$\oplus_{i\le0}D_i/D_{i+1}=U(\m)$ with evaluation at $x\in M$ being
the universal enveloping algebra of $\m_x$ (naturally graded).

Let $\Delta_i^\perp$ be the annihilators of the distribution $\Delta_i$
forming the decreasing filtration $\Omega^1_i=\Gamma(\Delta_{i-1}^\perp)$
of the space of 1-forms:
 $$
0\subset\Omega^1_\nu\subset\Omega^1_{\nu-1}\subset\dots\subset\Omega^1_2\subset\Omega^1_1=\Omega^1(M).
  $$
We have $\Omega^1_i/\Omega^1_{i+1}=\Gamma(\g_{-i}^*)$.
From these two filtrations we get the induced decreasing filtration of the module
 $$\mathfrak{D}=\op{Diff}(M;T^*M)=\op{Diff}(M)\ot_{C^\infty(M)}\Omega^1(M)$$
of 1-form valued differential operators on $M$ (here $C^\infty(M)$ acts on $D$ from the right,
$(\Delta,f)\mapsto\Delta\circ f$, $f\in C^\infty(M)$, $\Delta\in D$):
 $$
0\subset\mathfrak{D}_\nu=\Omega^1_\nu\subset\dots\subset
\mathfrak{D}_s=\sum_{i+j=s}D_i\ot_{C^\infty(M)}\Omega^1_j
\subset\mathfrak{D}_{s-1}\subset\dots
 $$
This filtration has the associated graded space, where every graded piece
$\mathfrak{D}_k/\mathfrak{D}_{k+1}$ is the space of sections of the vector bundle
 $$
\mathfrak{d}_k=\sum_{r+j=k}U(\m)_r\ot\g_{-j}^*,\ \quad -\infty<k\le\nu.
 $$
In other words we can embed
 $$
\mathfrak{d}_k\subset\sum_{r+j=k}\sum_{i_1+\dots+i_s=j}\prod_{t=1}^s
\g_{-r}^*\ot\g_{i_1}\ot\dots\ot\g_{i_t};
 $$
the equality is obtained by imposing on the right hand side the identifications from the universal
enveloping algebra $U(\m)$. Therefore $\oplus\mathfrak{d}_k$ is a module over $U(\m)$.

The natural action $F_{-1}\ot\mathfrak{D}_k\to\mathfrak{D}_{k-1}$ (composition) induces
the bundle map $\delta^*:\g_{-1}\ot\mathfrak{d}_{k+1}\to\mathfrak{d}_k$.

Consider the symmetry algebra of the distribution $\Delta$
(it can be sheafified and considered locally):
 $$
\mathcal{S}=\{X\in\D(M):L_X(\Delta)\subset\Delta\}=
\{\a(L_Xv)=0:\forall v\in\Gamma(\Delta),\a\in\Gamma(\Delta^\perp)\}.
 $$
Let $e_i^j$ be a basis of $TM$ such that $\Delta_s=\langle e_i^j:1\le j\le s\rangle$, $1\le s\le\nu$.
The dual co-basis $\a^p_q$ given by $\a^p_q(e_i^j)=\d^p_i\d^j_q$ yields
$\Delta_s^\perp=\langle\a^i_j:j>s\rangle$. We can write $\1=e_i^j\ot\a^i_j$.

Let $\a^p_q([e_i^j,e_k^r])=c_{ikq}^{jrp}$ denote the structure functions. Since
$L_v(X)=L_v(\a_j^i(X)e_i^j)=L_v(e_i^j)\a_j^i(X)+L_v(\a_j^i(X))e_i^j$ we obtain the defining
relation of $\mathcal{S}$ for $\a=\a_q^p$ ($q>1$) and $v=e^1_r$ in such form:
$\a_q^p(L_{e^1_r}(\1(X)))=0$ $\Leftrightarrow$ $\Box_{qr}^p(X)=0$, where
 $$
\Box_{qr}^p=c_{riq}^{1jp}\cdot\a_j^i+L_{e^1_r}\circ\a_q^p.
 $$
Thus the Lie equation $\mathfrak{Lie}(\Delta)$ defining $\mathcal{S}$ is given by
the linear differential operators $\Box_{qr}^p\in\mathfrak{D}$ ($q>1$):
 $$
\mathfrak{Lie}(\Delta)=\{[X]_x^1\in J^1(TM):\Box_{qr}^p(X)_x=0\}.
 $$

Notice that the Lie equation, as formulated, is not formally integrable (not in involution).
For instance, the compatibility conditions add the relations $L_X(\Delta_s)\subset\Delta_s$
for $s\le\nu$ and also $L_XK=0$, where $K$ is the curvature of the geometry.
Higher order equations can also appear as compatibilities.

Similarly, if $\mathcal{F}$ is a filtered structure, it is given by the Lie equation
$\mathfrak{Lie}(\mathcal{F})\subset J^k(TM)$ for some number $k$. This latter depends on
the reductions used to define $\mathcal{F}$ (the equations of $\mathfrak{Lie}(\Delta)$,
together with the compatibility conditions, guarantee that $X$ naturally lifts as
a vector field to $\GG_i$, then we impose the equation that the flow of $X$ preserves the
reductions). We keep denoting the defining differential operators by $\Box$.

Let $\E$ be the completion of the equation $\mathfrak{Lie}(\mathcal{F})$ to involution
obtained by the prolongation-projection method \cite{KLV,KL}. This equation is given by the differential
relations from the $D$-module (= module over differential operators $\op{Diff}(M)$)
$D[\Box]\subset\mathfrak{D}$ generated by the operators $\Box$.

We unite the obtained data into the following diagram, where all
horizontal arrows are monomorphisms and all vertical sequences are exact
(this defines the modules $Q[\Box]$):
 $$\begin{CD}
@. \dots @. \dots @. \dots @. \\
@. @VVV @VVV @VVV \\
0 @>>> D[\Box]_{1-i} @>>> \mathfrak{D}_{1-i} @>>> Q[\Box]_{1-i} @>>> 0 \\
@. @VVV @VV{\pi_{i,i-1}}^*V @VVV \\
0 @>>> D[\Box]_{-i} @>{\rho_i^*}>> \mathfrak{D}_{-i} @>>> Q[\Box]_{-i} @>>> 0 \\
@. @VVV @VVV @VVV \\
@. \dots @. \dots @. \dots @. \\
 \end{CD}$$

Recall that given a module $\mathfrak{M}$ over the algebra $C^\infty(M)$, which is
geometric (this means $\cap_{x\in M}(\mu_x\mathfrak{M})=0$; $\mu_x\subset C^\infty(M)$ are the maximal ideals)
and finitely generated, its property being projective is equivalent to
isomorphism of $\mathfrak{M}$ to the module of sections of a bundle. This is a smooth version
of the Serre-Swan theorem, see \cite{Sw,JN}. Localization (restriction) to an open set $U\subset M$ is
$\mathfrak{M}_U=\mathfrak{M}\otimes_{C^\infty(M)}C^\infty(M)/\mu_U$, where $\mu_S$ is the ideal
of functions vanishing on $S$; in other words, we change the ring to
$C^\infty(\bar{U}):=C^\infty(M)/\mu_U=\{f|_U:f\in C^\infty(M)\}\subset C^\infty(U)$. If all evaluations
$\mathfrak{M}_x=\mathfrak{M}\otimes_{C^\infty(M)}C^\infty(M)/\mu_x$, $x\in U$ have the same
rank, then by Nakayama's lemma $\mathfrak{M}_U$ is projective. Since geometricity and finite-generation
for our modules are given by construction, we obtain the associated vector bundles.

Namely, the modules of the above commutative diagram give rise to the vector bundles
$bD[\Box]$ (left column) and $\E^*$ (right column) over the set of regular points in $M$
(the middle column, consisting of projective modules, obviously corresponds to
sections of vector bundles over the whole $M$).
Here we call a point $x\in M$ {\it regular up to order $k$\/} if in a neighborhood $U\ni x$
the modules $Q([\Box])_{-i}|_U$ are projective for all $i\le k$;
a point $x\in M$ is {\it regular\/} if it is regular up to order $k$ for any $k$.

 \begin{prop}
If the filtered structure is of finite type or is analytic, then the set of regular points
for $\E$ is open dense in $M$. In general, the set of regular points up to order $k$ is
open dense in $M$ for every $k$; consequently the set of regular points has second category.
 \end{prop}

 \begin{proof}
For the structure $\mathcal{F}$ of finite type the number of added equations
during the prolongation-projection is finite, cf. Proposition \ref{prop1}, whence the claim.
By the same reason, in the general case the number of added equations
during prolongation-projection is finite up to any pre-fixed order $k$.
For the structure $\mathcal{F}$ of infinite analytic type the claim follows from
Malgrange's proof of Cartan-Kuranishi theorem \cite{Ma} that can be adapted to the filtered context.
 \end{proof}

We are going to realize the dual bundles to those arising from the above commutative diagram.
To this end let $J(TM)_i$ be the space of $i$-weighted jets
w.r.t. the weights introduced above ($i\ge-\nu$)\footnote{In the case $\Delta=TM$ we have $\nu=1$
and the whole jet co-filtration should be shifted by $+1$ to match the Spencer machinery \cite{Sp};
for  gradation of differential operators also the sign should reverse.}, which
is defined as the bundle over $M$ with the fiber at $x$ being quotient of the algebra of
$\mathcal{D}(M)$ by the subspace of those vector fields $X$ that
satisfy $\nabla(X)_x=0$ for all $\nabla\in\mathfrak{D}_{-i}$.
In terms of decomposition $X=\sum f_{jl}Z_{jl}$ by a basis $Z_{jl}\in\Delta_j\setminus\Delta_{j-1}\subset TM$
from \cite{K1} this means that the coefficients satisfy $f_{jl}\in\mu_{\Delta,x}^{i+j}$.

Let $\E_i\subset J(TM)_i$ be the set of points annihilated by the operators $D[\Box]_{-i}\subset D[\Box]$
of order $\ge-i$ (as above $\E$ is the completion of $\mathfrak{Lie}(\mathcal{F})$).
Over the connected components of the (open and dense) set of regular points
(where restrictions of the considered modules are projective) these are linear subbundles
and by duality we get the following commutative diagram, where the vertical lines
are exact and the horizontal arrows are epimorphisms (projections):
 $$\begin{CD}
@. \dots @. \dots @. \dots @. \\
@. @VVV @VVV @VVV \\
0 @>>> \E_i @>>> J(TM)_i @>{\rho_i}>> bD[\Box]^*_i @>>> 0 \\
@. @VVV @VV{\pi_{i,i-1}}V @VVV \\
0 @>>> \E_{i-1} @>>> J(TM)_{i-1} @>>> bD[\Box]^*_{i-1} @>>> 0 \\
@. @VVV @VVV @VVV \\
@. \dots @. \dots @. \dots @. \\
 \end{CD}$$
This geometrizes the $D$-modules corresponding to the Lie equation.

\section{Proof of Theorem \ref{thm1} and other applications}\label{S5}

 \begin{prop}
If the filtered structure $\mathcal{F}$ is analytic or of finite type, then
the space of solutions of $\mathfrak{Lie}(\mathcal{F})$ is $\mathcal{S}=\E_\kappa$
(for the infinite type $\kappa=+\infty$ and $\E_\kappa=\lim\limits_{i\to+\infty}\E_i$).
 \end{prop}

 \begin{proof}
In the finite type case the claim follows by finite jet-determination of a solution according to
\cite[Theorem~8]{K1}. For the analytic structure $\mathcal{F}$ the claim follows by
(infinite) jet-determination of a solution according to the filtered Cartan-K\"ahler theorem \cite{M2}.
 \end{proof}

It is enough to prove the main claim for structures of the type prescribed in this Proposition.
For general smooth filtered structures a symmetry (solution of the Lie equation) may not be
defined by its jet at a point, but in this case we only need the fact that the set of regular
points of $M$ is the intersection of the sets of points $x$ regular up to order $k$ by all $k$.

 \begin{proof}[Proof of Theorem \ref{thm1}]
Let $x\in M$ be a regular point (for the Lie equation).
The map $\mathcal{S}\to\E_\kappa$ associates to a symmetry $X$ its jet $[X]_x^\kappa$ at $x$.
It is an isomorphism for finite type or analytic structures.

Thus we obtain the decreasing filtration on $\mathcal{S}$ depending on the point $x$:
$\mathcal{S}^j=\op{Ker}(\mathcal{S}\to\E_{j-1})$.
The corresponding graded algebra coincides with the symbol of $\E$:
 $$
\mathfrak{s}_i=\mathcal{S}^i/\mathcal{S}^{i+1}=\op{Ker}(\E_i\to\E_{i-1}).
 $$
In particular $\mathfrak{s}_i\subset\g_i(x)$ and the claim  $\dim\mathcal{S}\le\dim\g(x)$ follows.

In the case of general type the right hand side of this inequality is $+\infty$, so the
second claim is void.

Finally, for finite type systems the first claim of the theorem implies the inequality
$\dim\mathcal{S}\leq\op{ess.}\inf\dim\g(x)$,
where the essential infimum is the supremum of $\inf_U\dim\g(x)$ by all open dense sets
$U\subset M$ (because the set of regular points is open and dense).
This latter equals to $\inf_M\dim\g(x)$ due to upper semi-continuity of $\g(x)$
(for the latter claim see \cite[Lemma~6]{K1}). The last claim of the theorem follows.
 \end{proof}

 \begin{rk}
The filtration on $\mathcal{S}$ coincides with the one from \cite{CN} in the case of parabolic geometries
and from \cite{K1} for non-holonomic distributions.
 \end{rk}

We can also prove the following claim that coincides with Proposition 4.1 of \cite{M1}
in the case of the standardly filtered Lie algebras and with Proposition 4.2.2 of \cite{KT}
in the case of filtration associated with parabolic geometries
(our present proof is different from both references).

 \begin{theorem}
The above embedding $\mathfrak{s}_i\subset\g_i$ satisfies:
$[\mathfrak{s}_i,\g_{-1}]\subset\mathfrak{s}_{i-1}$.
 \end{theorem}

 \begin{proof}
Let $\mathfrak{b}$ denote the symbol bundle of the $Q[\Box]$-modules over the regular set $U$:
$\Gamma(\mathfrak{b}_{-i})=Q[\Box]_{-i}/Q[\Box]_{1-i}$. At a regular point $x$:
$\mathfrak{s}_i\simeq\mathfrak{b}_{-i}^*(x)$. Indeed, from the commutative diagram above: $\mathfrak{s}_i=\op{Ker}(\pi_{i,i-1})\cap\op{Ker}(\rho_i)$
and so $\mathfrak{s}_i^\perp=\op{Im}(\pi_{i,i-1}^*)+\op{Im}(\rho_i^*)$. Therefore
$\mathfrak{s}_i^*=\mathfrak{D}_{-i}/[\op{Im}(\pi_{i,i-1}^*)+\op{Im}(\rho_i^*)]$.

By $D$-module property $F_{-1}\ot\mathfrak{D}_{1-i}\to\mathfrak{D}_{-i}$ and
$F_{-1}\ot D(\Box)_{1-i}\to D(\Box)_{-i}$, whence we conclude the action
$\d^*:\g_{-1}\ot\mathfrak{b}_{1-i}\to\mathfrak{b}_{-i}$ (at regular points).
By dualization $\d:\mathfrak{s}_i\to\g_{-1}^*\ot\mathfrak{s}_{i-1}$ or
$\d:\g_{-1}\ot\mathfrak{s}_i\to\mathfrak{s}_{i-1}$.
This latter corresponds to the bracket in $\g$ and hence our claim is proved.
 \end{proof}

Notice that the claim is stronger than the natural pairing
$[\mathfrak{s}_i,\mathfrak{s}_{-1}]\subset\mathfrak{s}_{i-1}$ since we only have
$\mathfrak{s}_{-1}\subset\g_{-1}$ and this distinction is crucial for \cite{KT}.

Finally let us apply the above results to jet-determinacy. A symmetry $X\in\mathcal{S}$
is called $s$-jet determined at $x\in M$ if $[X]_x^{s-1}=0$, but $[X]_x^s\ne0$
(here $[X]_x^k$ is the $k$-jet of $X$ in the standard, not weighted, jet-filtration).

 \begin{theorem}
If the Lie equation $\mathfrak{Lie}(\mathcal{F})$ has finite type, i.e. $\E_\kappa=0$ for some $\kappa<+\infty$,
then any symmetry is determined by a finite jet.
 \end{theorem}

The same statement also holds for automorphisms of a filtered geometry.

 \begin{proof}
By \cite{K1} (see formulae (9), (10) and those in between) if the field $X$ is $s$-jet determined
and belongs to the $i$-th filtration space $\mathcal{S}^i$ ($\Ll^i$ in loc.cit.), and so is mapped
to a non-zero element in the graded subspace $\g_i$ ($i\ge0$) of the Tanaka algebra
$\g=\g_{-\nu}\oplus\dots\oplus\g_i\oplus\dots$, then $i/\nu+1\le s\le i+1$.
 \end{proof}

\medskip

\textsc{Acknowledgment.} I am grateful to Dennis The for stimulating discussions.


 \vspace{-5pt} \hspace{-20pt} {\hbox to 12cm{ \hrulefill }}\vspace{-1pt}

{\footnotesize \hspace{-10pt} Institute of Mathematics and
Statistics, University of Troms\o, Troms\o\ 90-37, Norway.

\hspace{-10pt} E-mail: \quad boris.kruglikov\verb"@"uit.no} \vspace{-1pt}

\end{document}